\documentclass[12pt,oneside,english]{amsart}
\usepackage[T1]{fontenc}
\usepackage[utf8]{inputenc}
\usepackage{geometry}
\usepackage{fancyhdr}
\usepackage{color}
\usepackage{amsthm}
\usepackage{amssymb}
\usepackage{mathdots}
\usepackage{esint}

\providecommand{\tabularnewline}{\\}

\makeatletter
\numberwithin{equation}{section}
\numberwithin{figure}{section}
\theoremstyle{plain}
\newtheorem{thm}{\protect\theoremname}
\theoremstyle{plain}
\newtheorem{lem}[thm]{\protect\lemmaname}
\newtheorem{prop}[thm]{\protect\propositionname}
\newtheorem{cor}[thm]{\protect\corollaryname}
\usepackage[english]{babel}
\usepackage{fancyhdr}
\providecommand{\lemmaname}{Lemma}
\providecommand{\propositionname}{Proposition}
\providecommand{\theoremname}{Theorem}
\providecommand{\corollaryname}{Corollary}
\newcommand{\abs}[1]{\ensuremath{|#1|}}

\newcommand{\Abs}[1]{\ensuremath{\left|#1\right|}}

\newcommand{\norm}[2]{\ensuremath{|\!|#1|\!|_{#2}}}

\newcommand{\Norm}[2]{\ensuremath{\left|\!\left|#1\right|\!\right|_{#2}}}

\renewcommand{\d}[1]{\ensuremath{\textnormal{d}#1}}

\newcommand{\cC}{\mathcal{C}}

\newcommand{\cO}{\mathcal{O}}

\makeatother

\usepackage{babel}
  \providecommand{\lemmaname}{Lemma}
\providecommand{\theoremname}{Theorem}

\begin{document}

\title{$l_{p}$-norms of Fourier coefficients of powers of a Blaschke factor}

\author{Oleg Szehr}

\address{Department of Mathematics and Mechanics, Saint Petersburg State University,
28, Universitetski pr., St. Petersburg, 198504, Russia.}

\email{oleg.szehr@posteo.de}

\author{Rachid Zarouf}

\address{Aix-Marseille Université, CNRS, Centrale Marseille, LATP, UMR 7353,
13453 Marseille, France.}

\email{rachid.zarouf@univ-amu.fr}

\address{Department of Mathematics and Mechanics, Saint Petersburg State University,
28, Universitetski pr., St. Petersburg, 198504, Russia.}

\thanks{The work is supported by Russian Science Foundation grant 14-41-00010}

\email{rzarouf@gmail.com}

\begin{abstract}
We determine the asymptotic behavior of the $l_{p}$-norms of the sequence of Taylor coefficients of $b^{n}$,
where $b=\frac{z-\lambda}{1-\bar{\lambda}z}$ is an automorphism of the unit disk, $p\in[1,\infty]$, and $n$ is large.
It is known that in the parameter range $p\in[1,2]$ a sharp upper bound
\begin{align*}
\Norm{b^{n}}{l_{p}^A}\leq C_{p}n^{\frac{2-p}{2p}}
\end{align*}
holds. In this article we find that this estimate is valid even when $p\in[1,4)$. We prove that 
\begin{align*}
\Norm{b^{n}}{l_{4}^A}\leq C_{4}\left(\frac{\log n}{n}\right)^{\frac{1}{4}}
\end{align*}
and for $p\in(4,\infty]$ that
\begin{align*}
\Norm{b^{n}}{l_{p}^A}\leq C_{p}n^{\frac{1-p}{3p}} & .
\end{align*}
We prove that our upper bounds are sharp as $n$ tends to $\infty$ i.e.~they have the correct asymptotic $n$ dependence.
\end{abstract}
\maketitle

\section{\label{sub:Introduction}Introduction}

We denote by $\mathbb{D}=\{z:\:\abs{z}<1\}$ the open unit disk in
the complex plane and by $\partial\mathbb{D}$ its boundary. For a
given $\lambda\in\mathbb{D}$ we denote by 
\[
b=b_{\lambda}=\frac{z-\lambda}{1-\bar{\lambda}z}
\]
the elementary Blaschke factor corresponding to $\lambda$. Clearly
$\abs{b_{\lambda}(z)}=1$ is equivalent to $z\in\partial\mathbb{D}$.
For any $n$ we have that $B=b^{n}$ is a bounded, holomorphic on
$\mathbb{D}$ and as such posses a natural identification with its
boundary behavior on $\partial\mathbb{D}$~\cite{NN}. It is well
known that the Taylor- and Fourier- coefficients of such functions
can be identified~\cite{NN} and we will use these terms interchangeably
in what follows. Let $B=\sum_{k\geq0}\widehat{B}(k)z^{k}$ denote the
Taylor expansion of $B=b^{n}$. We write 
\[
\Norm{B}{l_{p}^{A}}^{p}:=\norm{\widehat{B}}{l_{p}}:=\sum_{k\geq0}\abs{\hat{B}(k)}^{p}
\]
for the usual $l_{p}$-norm of the sequence of Taylor coefficients
of $B$. In the limit of large $p$ we set $\Norm{B}{l_{\infty}}:=\sup_{k}\abs{\widehat{B}(k)}$.
We observe that our $l_{p}$-norms only depend on the absolute values
$\abs{\hat{B}(k)}$ and we have 
\[
\Norm{\left(\frac{z-\lambda}{1-\bar{\lambda}z}\right)^{n}}{l_{p}^{A}}=\Norm{\left(\frac{z-\abs{\lambda}}{1-\abs{\lambda}z}\right)^{n}}{l_{p}^{A}}.
\]
In what follows we therefore assume $\lambda\in(0,1)$. We use the following notation: for two positive functions $f,g\: :\: \mathbb{C}\rightarrow\mathbb{R}^+$ we say that $f$ is dominated by $g$, denoted by $f\lesssim g$, if there is a constant $c>0$ such that $f\leq cg;$ and we say that $f$ and $g$ are comparable, denoted by $f\asymp g$, if both $f\lesssim g$ and $g\lesssim f$. In this article
we seek to determine the asymptotic behavior of the norm 
\[
\Norm{B}{l_{p}^{A}}^{p}\ \textnormal{for}\ p\in[1,\infty]
\]
in the limit of large $n$. This appears to be a relatively well-studied
topic. 
 J-P.~Kahane~\cite{JK} has shown that $\Norm{B}{l_{1}^{A}}\asymp c_{1}\sqrt{n}$.
The sharp numerical constant $c_{1}$, i.e. such that $\Norm{B}{l_{1}^{A}}/\sqrt{n}\rightarrow c_{1}$ as $n$ tends to $\infty$, was computed in~\cite{DG} but the proof was
carried out more precisely in \cite{BHl}. The discussion of $l_{p}$-norms
for $p\neq1$ occured in~\cite{BS}, where the asymptotic behavior
\[
\Norm{B}{l_{p}^{A}}\asymp c_{p}n^{\frac{2-p}{2p}}\ \textnormal{for}\ p\in[1,2]
\]
is derived. Our aim is to extend this discussion to the whole interval
$p\in[1,\infty]$. For us this is motivated by a line
of research that aims at Schäffer's question~\cite{JG,GMP}. We intend to sharpen the results of~\cite{GMP} using precise estimates for $\Norm{b^{n}}{l_{\infty}^{A}}$. The results appear in a forthcoming article~\cite{SZ2}. In~\cite{BS} the authors study the composition operator defined by
the relation $(comp_{b}(f))(z):=f(b(z))$. To assess if $comp_{b}$
is a bounded linear operator from one Banach space of analytic functions
into another it is often enough to know the asymptotic behavior of
$\Norm{b^n}{}$. For example the closed graph theorem shows that
$comp_{b}$ is bounded from $l_{1}^{A}$ to $l_{p}^{A}$ iff the norms
$\Norm{B}{l_{p}^{A}}$ are uniformly bounded. Our main result is as follows.
\begin{thm}
\label{main}Let $B=b_{\lambda}^{n}$ denote a Blaschke product, let
$\lambda\in(0,1)$ and $n\geq1$. Then for the $l_{p}$-norms of the
Taylor coefficients of $B$ we have the following asymptotic behavior
\[
\Norm{B}{l_{p}^{A}}\asymp\begin{cases}
c_{p}\: n^{\frac{2-p}{2p}} & \mbox{if}\ p\in(1,4)\\
c_{4}\:\left(\frac{\log n}{n}\right)^{\frac{1}{4}} & \mbox{if}\ p=4\\
c_{p}\: n^{\frac{1-p}{3p}} & \mbox{if}\ p\in(4,\infty]
\end{cases}
\]
for some constants $c_{p}$ depending on $p$ and $\lambda$. 
\end{thm}
Our proof is based on a detailed analysis of the asymptotic growth of the Taylor coefficients $\widehat{B}(k)$ both with respect to $k$ and $n$. As it turns out this analysis is delicate. Holomorphy of $B$ implies that for any \emph{fixed} $n$ the coefficients $\widehat{B}(k)$ decay exponentially when $k$ grows large. Similarly, it is not difficult to see that at any~\emph{fixed} $k$ the coefficient $\widehat{B}(k)$ decays exponentially with $n$, see Proposition~\ref{main_Lemma} below. The interesting behavior, which is relevant for determining the norms $\Norm{B}{l_{p}^{A}}$, therefore occurs when $k=k(n)$ is a sequence. As $n$ grows large the region of values $k$ that provide the dominating contribution to $\Norm{B}{l_{p}^{A}}$ can change. For instance in case of $\Norm{B}{l_{\infty}^{A}}=\sup_{k\geq0}\abs{\widehat{B}(k)}$ we can guess that the supremum will be achieved on a coefficient whose index $k$ depends on $n$. The question is now, what is actually the right sequence $k=k(n)$ such that $\Norm{B}{l_{\infty}^{A}}$ is achieved. More generally, for our exercise it is crucial to identify for each $p$, which values of $k$ provide the dominating contribution to $\Norm{B}{l_{p}^{A}}$. We therefore decompose the set of values for $k$ into $n$-depending \lq\lq{}intervals\rq\rq{} and show that the regions of $k$ that provide the dominating contributions to $\Norm{B}{l_{p}^{A}}$ depend on $p$ and $\lambda\in(0,1)$.
This fact is one of the main findings of the article at hand and was not observed in preceding publications. In this fact lies also the reason for the structure of the asymptotic behavior provided in Theorem~\ref{main}. Depending on whether $p\in(1,4)$ or $p\in(4,\infty)$ the dominating contribution stems from different regions of $k$ resulting in the differing asymptotics. The dependence on $\lambda$ can be described in terms of the \lq\lq{}critical\rq\rq{} values $\alpha_0=\frac{1-\lambda}{1+\lambda}$ and $\alpha_0^{-1}$, which will be stationary points for the expansion of integrals in our asymptotic analysis. For now a simple way of understanding their critical nature is to view them as values that identify the slowest decay for $\widehat{B}(k)$ in the sense that at $k=\lfloor\alpha_{0}n\rfloor$ and $k=\lfloor\alpha_{0}^{-1}n\rfloor$ we observe the slowest decay of $\widehat{B}(k(n))$ when $n$ grows large. A summary of decay rates of $\abs{\widehat{B}(k)}$ is provided in Table~\ref{PussWideOpen}.
\begin{figure}[ht!]\label{PussWideOpen}
\centering

\begin{tabular}{|c|c|c|}
\hline 
Values of $k(n)$ in interval & Decay of $\widehat{B}(k)$ & Region\tabularnewline
\hline 
\hline 
$[0,\,\alpha n]$ & exponential & I\tabularnewline
\hline 
$(\alpha n,\,\alpha_{0}n-n^{1/3}]$ & $\frac{1}{\alpha_{0}n-k}$ & II\tabularnewline
\hline 
$[\alpha_{0}n-n^{1/3},\,\alpha_{0}n+n^{1/3}]$ & $\frac{1}{n^{1/3}}$ & III\tabularnewline
\hline 
$[\alpha_{0}n+n^{1/3},\,\alpha_{0}^{-1}n-n^{1/3}]$ & $\frac{1}{n^{1/2}\left(\frac{k}{n}-\alpha_{0}\right)^{1/4}\left(\alpha_{0}^{-1}-\frac{k}{n}\right)^{1/4}}$ & IV\tabularnewline
\hline 
$[\alpha_{0}^{-1}n-n^{1/3},\,\alpha_{0}^{-1}n+n^{1/3}]$ & $\frac{1}{n^{1/3}}$ & V\tabularnewline
\hline 
$[\alpha_{0}^{-1}n+n^{1/3},\,\alpha^{-1}n)$ & $\frac{1}{k-\alpha_{0}^{-1}n}$ & VI\tabularnewline
\hline 
$[\alpha^{-1}n,\,\infty)$ & exponential & VII\tabularnewline
\hline
\end{tabular}
\caption{Illustration of decay of $\widehat{B}(k)$ as a function of $k=k(n)$. We set $\alpha_0=\frac{1-\lambda}{1+\lambda}$ and fix arbitrary $\alpha\in(0,\alpha_0)$.}
\end{figure}

In this article we split the discussion of $\widehat{B}(k)$ and $\Norm{B}{l_{p}^{A}}$ conceptually in the derivation of upper and lower estimates.
In Section~\ref{sec:Upper-estimates} we derive upper estimates $\abs{\widehat{B}(k)}$ and compute the resulting upper estimates for $\Norm{B}{l_{p}^{A}}$. In Section \ref{sec:Lower-bounds} we prove the asymptotic sharpness of our upper estimates. Our proof of upper bounds on $\widehat{B}(k)$ will be based on a well-known Van der Corput type estimates, Lemma~\ref{VDCP}. It turns out that in the interval $p\in[1,4)$ sharpness follows from a simple application of Hölder's inequality. The proof of sharpness in the range $p\in[4,\infty]$, however, requires new methods. A core step will be the introduction and development of the so-called
\textit{uniform method of stationary phase} \cite{VB,CFU}, which we employ to derive an asymptotic expansion of $\hat{B}(k)$ when $k$ is near to $\alpha_0^{-1}n$. This method will provide the sharpness of our upper bound when $p\in(4,\infty]$.

\section{\label{sec:Upper-estimates}Upper estimates}

To prove the upper bounds in Theorem \ref{main} we estimate the norm of the $k$-th Taylor coefficient of the $n$-th power of $b_{\lambda}$. Summing the individual coefficients will provide the desired bounds for Theorem~\ref{main}. 

\begin{prop} \label{main_Lemma} Let $B=b_{\lambda}^{n}$ with $\lambda\in(0,1)$
and $n\geq1$. Set $\alpha_{0}:=\frac{1-\lambda}{1+\lambda}$
and choose a fixed $\alpha\in(0,\alpha_{0})$. In the following
we consider sequences $k=k(n)$ and all assertions are meant to hold
for large enough $n$.
\begin{enumerate}
\item If $k/n\leq\alpha$ then $\abs{\widehat{B}(k)}$ decays exponentially
as $n$ tends to $\infty$. Similarly if $k/n\geq\alpha^{-1}$ then
$\abs{\widehat{B}(k)}$ decays exponentially as $n$ tends to $\infty$.
\item If $k/n\in(\alpha,\alpha_{0}-n^{-2/3})\cup(\alpha_{0}^{-1}+n^{-2/3},\alpha^{-1})$
then 
\[
\Abs{\widehat{B}(k)}\lesssim4\max\left\{ \Abs{\frac{1}{\alpha_{0}n-k}},\Abs{\frac{1}{\alpha_{0}^{-1}n-k}}\right\} .
\]

\item If $k/n\in[\alpha_{0}-n^{-2/3},\alpha_{0}+n^{-2/3})\cup(\alpha_{0}^{-1}-n^{-2/3},\alpha_{0}^{-1}+n^{-2/3}]$
then 
\[
\abs{\widehat{B}(k)}\lesssim\frac{1}{n^{1/3}}.
\]

\item If $k/n\in(\alpha_{0}+n^{-2/3},\alpha_{0}^{-1}-n^{-2/3})$, then 
\[
\abs{\widehat{B}(k)}\lesssim\max\left\{ \Abs{\frac{1}{n^{1/2}\left(\alpha_{0}-\frac{k}{n}\right)^{1/4}}},\Abs{\frac{1}{n^{1/2}\left(\alpha_{0}^{-1}-\frac{k}{n}\right)^{1/4}}}\right\} .
\]

\end{enumerate}
\end{prop}
We begin with a well-known lemma due to Van Der
Corput. It will be the key ingredient for the upper estimates of Proposition~\ref{main_Lemma}.
\begin{lem}
\label{VDCP} Let $g$ be a continuously differentiable real
function on $[a,b]\subset\mathbb{R}$, such that $g$ and $g'$ are monotone and $g'$ does not vanish on $[a,b].$ Then 
\[
\Abs{\int_{a}^{b}e^{ig(t)}{\rm d}t}\leq\frac{2}{\abs{g'(a)}}+\frac{2}{\abs{g'(b)}}.
\]
\end{lem}
\begin{proof}
Integration by parts shows that

\begin{align*}
\int_{a}^{b}e^{ig(t)}{\rm d}t = \int_{a}^{b}e^{ig(t)}\frac{g'(t)}{g'(t)}{\rm d}t
  =  \left[e^{ig(t)}\frac{1}{g'(t)}\right]_{a}^{b}-\int_{a}^{b}e^{ig(t)}\left(\frac{1}{g'(t)}\right)'{\rm d}t.
\end{align*}
This provides the rough upper estimate
\[
\Abs{\int_{a}^{b}e^{ig(t)}{\rm d}t}\leq\frac{1}{\abs{g'(a)}}+\frac{1}{\abs{g'(b)}}+\int_{a}^{b}\frac{\abs{g''(t)}}{(g'(t))^{2}}{\rm d}t.
\]
Since $g'$ is monotone on $[a,b]$ we have either $g''\geq0$ or
$g''\leq0$ on $[a,b]$ and consequently 
\[
\int_{a}^{b}\frac{\abs{g''(t)}}{(g'(t))^{2}}{\rm d}t=\Abs{\int_{a}^{b}\left(\frac{1}{g'(t)}\right)'{\rm d}t}\leq\frac{1}{\abs{g'(a)}}+\frac{1}{\abs{g'(b)}}.
\] 
\end{proof}
To apply the lemma we rewrite $\widehat{b_{\lambda}^{n}}(k)$ in a convenient way. First $b_{\lambda}(e^{it})\in\partial\mathbb{D}$ for any $t\in(-\pi,\pi]$ and there exists a real
valued function $f_{\lambda}$ so that 
\[
b_{\lambda}(e^{it})=e^{if_{\lambda}(t)},\qquad\forall t\in\left(-\pi,\,\pi\right].
\]
Deriving the above equality with respect to $t$ we find 
\[
ie^{it}\frac{\lambda^{2}-1}{(1-\lambda e^{it})^{2}}=if_{\lambda}'(t)b_{\lambda}(e^{it})
\]
which shows that 
\[
f_{\lambda}'(t)=\frac{1-\lambda^{2}}{\abs{1-\lambda e^{it}}^{2}}=\frac{1-\lambda^{2}}{1+\lambda^{2}-2\lambda\cos(t)},\qquad \forall t\in\left(-\pi,\,\pi\right].
\]
For the Taylor coefficient with $n\geq1$ and $k\geq0$
we can write 
\begin{align*}
\widehat{b_{\lambda}^{n}}(k)=\frac{1}{2\pi}\int_{-\pi}^{\pi}b_{\lambda}^{n}(e^{it})e^{-ikt}{\rm d}t
=\frac{1}{2\pi}\int_{-\pi}^{\pi}e^{i(nf_{\lambda}(t)-kt)}{\rm d}t=\frac{1}{2\pi}\int_{-\pi}^{\pi}e^{ig(t)}{\rm d}t=\frac{1}{\pi}\Re\left\{\int_{0}^{\pi}e^{ig(t)}{\rm d}t\right\},
\end{align*}
where 
\[
g(t)=nf_{\lambda}(t)-kt\qquad t\in[0,\pi].
\]
Computing derivatives we find that $g'(0)=n\alpha_{0}^{-1}-k$,
$g'(\pi)=n\alpha_{0}-k$ and
\[
g''(t)=-\frac{2\lambda n(1-\lambda^{2})\sin(t)}{(1+\lambda^{2}-2\lambda\cos(t))^{2}}.
\]
This implies that $g'$ is strictly decreasing on $(0,\pi)$ with
\[
g'(0)=n\alpha_{0}^{-1}-k>g'(t)>n\alpha_{0}-k=g'(\pi).
\]

\begin{proof}[Proof of Proposition~\ref{main_Lemma}]
For simplicity we focus on the case where $k$ is closer to $n\alpha_{0}$ than to $n\alpha_{0}^{-1}$. The discussion in the alternative case is identical.

\begin{enumerate}
\item This is a direct application of \cite[Theorem 2, point (3)]{SZ}. We recapitulate the main steps for completeness. It is well known~\cite{JG} that for $z,w\in\mathbb{D}$ we have
upper and lower bounds on the elementary Blaschke factor as 
\begin{align*}
\frac{\abs{z}-\abs{w}}{1-\abs{z}\abs{w}}\leq\Abs{\frac{z-w}{1-\bar{w}z}}\leq\frac{\abs{z}+\abs{w}}{1+\abs{z}\abs{w}}.
\end{align*}
Fourier coefficients can be expressed using the usual contour integral
\begin{align*}
\abs{\widehat{B}(k)}=\Abs{\frac{1}{2\pi}\oint_{\partial\mathbb{D}}z^{-k-1}\left(\frac{z-\lambda}{1-\lambda z}\right)^{n}\d z}.
\end{align*}
For the magnitude of the integral we find that 
\begin{align}
\abs{\widehat{B}(k)}\leq\max_{\abs{z}=s}{\frac{\abs{b_{\lambda}^{n}(z)}}{{\abs{z}^{k}}}}=\begin{cases}
\frac{b_{\lambda}^{n}(s)}{{s^{k}}},\ s\in(1,1/\lambda)\\
\frac{\left(\frac{s+\lambda}{1+\lambda s}\right)^{n}}{{s^{k}}},\ s\in(0,1)
\end{cases}.\label{eq:}
\end{align}
If $k/n\geq\alpha^{-1}$ then there exists $s^*\in(1,1/\lambda)$ such that $\frac{b_{\lambda}(s^*)}{{{s^*}^{k/n}}}{\leq\frac{b_{\lambda}(s^*)}{{{s^*}^{\alpha^{-1}}}}}<1$~\cite{SZ}.
If $k/n{\leq}\alpha$ then there exists $s^*\in(\lambda,1)$ such that $\frac{\left(\frac{s^*+\lambda}{1+\lambda s^*}\right)}{{{s^*}^{k/n}}}{\leq\frac{\left(\frac{s^*+\lambda}{1+\lambda s^*}\right)}{{{s^*}^{{\alpha}}}}}<1$~\cite{SZ}.

\item If $k/n\in(\alpha,\alpha_{0}-n^{-2/3})$ then $g'(\pi)>0.$ In particular
$g'>0$ on $[0,\pi]$ and $g$ is strictly increasing while $g'$
is decreasing on this interval. Applying Lemma \ref{VDCP} we get
\begin{eqnarray*}
\abs{\widehat{b_{\lambda}^{n}}(k)} & \leq & \frac{2}{n\alpha_{0}-k}+\frac{2}{n\alpha_{0}^{-1}-k}\\
 & \leq & \frac{4}{n\alpha_{0}-k}.
\end{eqnarray*}

\item If $k/n\in[\alpha_{0} -n^{-2/3},\alpha_{0} +n^{-2/3})$ then $g'(\pi)=\alpha_{0}n-k$ might be positive or negative depending on the choice of $k$.
We fix a constant $c_{1}>0$ (independent of $n$) whose exact value is to be specified later. We split the integral
\[
\int_{0}^{\pi}e^{ig(t)}{\rm d}t=\int_{0}^{\pi-c_1n^{-1/3}}e^{ig(t)}{\rm d}t+\int_{\pi-c_1cn^{-1/3}}^{\pi}e^{ig(t)}{\rm d}t
\]
and notice that
\[
\Abs{\int_{\pi-c_1n^{-1/3}}^{\pi}e^{ig(t)}{\rm d}t}\leq c_1n^{-1/3}.
\]
To estimate the second integral we intend to apply the Van der Corput-type Lemma~\ref{VDCP}, which requires a lower estimate on $g'(t)$ for $t\in[0,\pi-c_1n^{-1/3}]$. To achieve this we expand the function $f'_\lambda$ in a neighborhood of $\pi$, which provides
\[
f_{\lambda}'(\pi-u)=\alpha_{0}+\frac{\lambda(1-\lambda)}{(1+\lambda)^{3}}u^{2}+\cO(u^{4})
\]
as $u$ tends to $0$. Hence for the decreasing function $g'$ we find that for large $n$ and $t\in[0,\pi-c_1n^{-1/3}]$ we have
\begin{align*}
g'(t)\geq g'(\pi-c_{1}n^{-1/3})& =nf_\lambda'(\pi-c_{1}n^{-1/3})-k\\
 & \geq\alpha_{0}n+ c_{1}^{2}\frac{\lambda(1-\lambda)}{(1+\lambda)^{3}}n^{1/3}+\cO(n^{-1/3})-k\\
 & \geq c_{1}^{2}\frac{\lambda(1-\lambda)}{(1+\lambda)^{3}}n^{1/3}-n^{1/3}+\cO(n^{-1/3})\\
 &\geq n^{1/3},
\end{align*}
where we made use of the assumption $k\leq\alpha_{0} n+n^{1/3}$ and have chosen appropriate $c_{1}=c_1(\lambda)>0$ for the last inequality. Applying Lemma \ref{VDCP} on $[0,\pi-c_1 n^{-1/3}]$
we obtain
\[
\Abs{\int_{0}^{\pi-c_1n^{-1/3}}e^{ig(t)}{\rm d}t}\leq4n^{-1/3}.
\]

\item If $k/n\in(\alpha_{0} +n^{-2/3},\alpha_{0}^{-1} -n^{-2/3})$ then the equation $g'(t)=0$ has exactly one solution $\varphi_{+}$ on $(0,\pi)$, see~\cite{SZ}. Direct computation shows that
\[
\abs{g''(\varphi_{+})}=k\left(\frac{k}{n}-\alpha_{0}\right)^{1/2}\left(\alpha_{0}^{-1}-\frac{k}{n}\right)^{1/2}.
\]
We choose $\delta=\delta(n)>0$ whose exact value is to be specified later. We split the integral
\[
\int_{0}^{\pi}e^{ig(t)}{\rm d}t=\int_{0}^{\varphi_+-\delta}e^{ig(t)}{\rm d}t+\int_{\varphi_+-\delta}^{\varphi_++\delta}e^{ig(t)}{\rm d}t+
\int_{\varphi_++\delta}^{\pi}e^{ig(t)}{\rm d}t
\]
and notice that
$$\Abs{\int_{\varphi_+-\delta}^{\varphi_++\delta}e^{ig(t)}{\rm d}t}\leq2\delta.$$
The remaining integrals are treated via Lemma~\ref{VDCP}. Since $g'$ is decreasing on $(0,\pi)$ and $g'(0)=\alpha_{0}n-k$ we have%
\begin{align*}
\Abs{\int_{0}^{\varphi_{+}-\delta}e^{ig(t)}{\rm d}t}\lesssim & \frac{1}{\abs{g'(0)}}+\frac{1}{\abs{g'(\varphi_{+}-\delta)}}=\frac{1}{\abs{n\alpha_{0}^{-1}-k}}+\frac{1}{\abs{g'(\varphi_{+}-\delta)}}.
\end{align*}%
As always we assume that $k$ is closer to $\alpha_{0} n$ so that $1/\abs{n\alpha_{0}^{-1}-k}\leq1/\abs{n\alpha_{0}-k}$ and we seek for a suitable lower bound for $\abs{g'(\varphi_{+}-\delta)}$. This is achieved as follows.
First we use the mean-value theorem for integrals to see that there is $s=s(n)$ with
$$\abs{g'(\varphi_{+}-\delta)}=\Abs{\int_{\varphi_{+}-\delta}^{\varphi_{+}}g''(t){\rm d}t}=\abs{g''(s)}\delta$$
for some $s\in(\varphi_{+}-\delta,\varphi_{+})$. By the mean-value theorem for differentiation there exists also $u=u(n)\in(s,\varphi_{+})$
such that
\begin{align*}
\frac{\abs{g''(\varphi_{+})-g''(s)}}{\abs{g''(\varphi_{+})}}=  \frac{(\varphi_{+}-s)\abs{g'''(u)}}{\abs{g''(\varphi_{+})}}&\lesssim \frac{\delta n}{k\left(\frac{k}{n}-\alpha_{0}\right)^{1/2}\left(\alpha_{0}^{-1}-\frac{k}{n}\right)^{1/2}}\\
&\lesssim \frac{\delta}{\left(\frac{k}{n}-\alpha_{0}\right)^{1/2}}.
\end{align*}
For the last inequality we use that $k/n\in(\alpha_{0} +n^{-2/3},\alpha_{0}^{-1}-n^{-2/3})$ is bounded from below. We also made use of the assumption that $k/n$ is closer to $\alpha_{0}$ than $\alpha_{0}^{-1}$, which implies that $\left(\frac{k}{n}-\alpha_{0}\right)^{-1/2}$ is bounded by a constant.
In particular assuming $\delta=\frac{1}{2\left(\frac{k}{n}-\alpha_{0}\right)^{1/4}n^{1/2}}$
we have $\delta\leq\frac{1}{2n^{1/3}}$ and
\begin{align*}
\abs{g''(s)} & \geq\abs{g''(\varphi_{+})}\cdot\Abs{1-\frac{\abs{g''(\varphi_{+})-g''(s)}}{\abs{g''(\varphi_{+})}}}\\
 & \gtrsim\abs{g''(\varphi_{+})}.
\end{align*}
In summary we find
\begin{align*}
\Abs{\int_{0}^{\varphi_{+}-\delta}e^{ig(t)}{\rm d}t}\lesssim & \frac{1}{k-\alpha_{0} n}+\frac{1}{\delta\abs{g''(\varphi_{+})}}\\
\lesssim & \frac{1}{k-\alpha_{0} n}+\frac{1}{\delta n\left(\frac{k}{n}-\alpha_{0}\right)^{1/2}}\\
\lesssim & \frac{1}{k-\alpha_{0} n}+\frac{1}{n^{1/2}\left(\frac{k}{n}-\alpha_{0}\right)^{1/4}}.
\end{align*}
A similar reasoning applies to $\int_{\varphi_{+}+\delta}^{\pi}e^{ig(t)}{\rm d}t$. We obtain in total
\begin{align*}
\Abs{\int_{0}^{\pi}e^{ig(t)}{\rm d}t}\lesssim & \delta+\frac{1}{k-\alpha_{0} n}+\frac{1}{n^{1/2}\left(\frac{k}{n}-\alpha_{0}\right)^{1/4}}\\
\lesssim & \frac{1}{k-\alpha_{0} n}+\frac{1}{n^{1/2}\left(\frac{k}{n}-\alpha_{0}\right)^{1/4}}\\
\lesssim & \frac{1}{n^{1/2}\left(\frac{k}{n}-\alpha_{0}\right)^{1/4}}\left(1+\frac{1}{n^{1/2}\left(\frac{k}{n}-\alpha_{0}\right)^{3/4}}\right)\\
\lesssim & \frac{1}{n^{1/2}\left(\frac{k}{n}-\alpha_{0}\right)^{1/4}}\left(1+\frac{1}{n^{1/2}\left(\frac{k}{n}-\alpha_{0}\right)^{3/4}}\right)\\
\lesssim & \frac{1}{n^{1/2}\left(\frac{k}{n}-\alpha_{0}\right)^{1/4}}
\end{align*}
which completes the proof. 
\end{enumerate}
\end{proof}

We prove the upper bound in Theorem \ref{main}. 
\begin{proof}[Proof of upper bound in Theorem~\ref{main}]
We set $\beta=\frac{\alpha_{0}+\alpha_{0}^{-1}}{2}$ and split the sum 
\[
\Norm{B}{l_{p}^{A}}^{p}=\sum_{k\leq\beta n}\abs{\widehat{B}(k)}^{p}+\sum_{ k>\beta n}\abs{\widehat{B}(k)}^{p}.
\]
For the proof we focus on the second sum, i.e.~we assume that $k$ is closer to $\alpha_0^{-1}n$ than to $\alpha_0n$. (This is for completeness of the exposition and complementary to the proof of Proposition~\ref{main_Lemma}, where we focus on $k$ closer $\alpha_0n$.) The discussion of the first sum is identical. Let $\alpha<\alpha_0$. We split the sum over $k>\beta n$ according to the regions of Proposition~\ref{main_Lemma}
\begin{align*}
&\sum_{k>\beta n}\abs{\widehat{B}(k)}^{p}=\\
&\left(\sum_{\beta n<k\leq \alpha_0^{-1}n-n^{1/3}}+\sum_{\alpha_0^{-1}n-n^{1/3}<k\leq\alpha_0^{-1}n+n^{1/3}}+\sum_{\alpha_0^{-1}n+n^{1/3}<k\leq\alpha^{-1} n}+\sum_{ \alpha^{-1} n < k}\right)\abs{\widehat{B}(k)}^{p}.
\end{align*}
We make use of the respective estimates of Proposition~\ref{main_Lemma} to bound the individual sums.
\begin{itemize}
\item We begin by the \lq\lq large values of $k$\rq\rq, where coefficients decay exponentially. We have
\[
\sum_{\alpha^{-1}n<k}\abs{\widehat{B}(k)}^{p}=\cO\left(n^{1/3}e^{-p(\alpha^{-1}-\alpha_{0}^{-1})n^{2/3}}\right).
\]
Using the first estimate in \eqref{eq:} we find
\[
\sum_{\alpha^{-1}n<k}\abs{\widehat{B}(k)}^{p}\leq\abs{b_{\lambda}(s)}^{pn}\frac{s{}^{-pN_{1}}}{1-s{}^{-p}}.
\]
We choose the radius $s=s_{n}=1+\frac{1}{n^{1/3}}$, which gives 
\[
\frac{b_{\lambda}(s_{n})}{s_{n}^{\alpha_{0}}}=1+\frac{\lambda(1+\lambda)}{3(1-\lambda)^{3}}\frac{1}{n}+\cO\left(\frac{1}{n^{4/3}}\right).
\]
Moreover 
\[
s_{n}{}^{-p}=1-\frac{p}{n^{1/3}}+\cO\left(\frac{1}{n^{2/3}}\right).
\]
In total we get
\begin{align*}
\sum_{\alpha^{-1}n<k}\abs{\widehat{B}(k)}^{p} & \leq\left(\frac{b_{\lambda}(s_{n})}{s_{n}^{\alpha_{0}^{-1}}}\right)^{pn}\frac{s_{n}{}^{-p(\alpha^{-1}-{\alpha_{0}^{-1}})n}}{1-s{}_{n}^{-p}}\\
 & \leq n^{1/3}e^{\frac{\lambda p(1+\lambda)}{3(1-\lambda)^{3}}}\frac{s_{n}{}^{-p(\alpha^{-1}-{\alpha_{0}}^{-1})n}}{p+\cO\left(\frac{1}{n^{1/3}}\right)}\\
 & =\cO\left(n^{1/3}e^{-p(\alpha^{-1}-\alpha_{0}^{-1})n^{2/3}}\right).
\end{align*}

\item We estimate 
\[
\sum_{\alpha_{0}^{-1}n+n^{1/3}<k\leq\alpha^{-1}n}\abs{\widehat{B}(k)}^{p}\lesssim\frac{c^{p}}{n^{p-1}}\frac{1}{n}\sum_{\alpha_{0}^{-1}n+n^{1/3}<k\leq\alpha^{-1}n}\frac{1}{\left({\frac{k}{n}-}\alpha_{0}^{-1}\right)^{p}}.
\]
We set $f(t)=\frac{1}{\left({t-}\alpha_{0}^{-1}\right)^{p}}$
and bound the Riemann sum 
\begin{align*}
\frac{1}{n}\sum_{\alpha_{0}^{-1}n+n^{1/3}<k\leq\alpha^{-1}n}f\left(\frac{k}{n}\right) & {\lesssim\int_{\alpha_{0}^{-1}+n^{-2/3}}^{\alpha^{-1}}f(t){\rm d}t}\\
 & {\lesssim\int_{n^{-2/3}}^{\alpha^{-1}-\alpha_{0}^{-1}}\frac{1}{u^{p}}{\rm d}u}.
\end{align*}
We find for $p=1$ 
\[
\sum_{\alpha_{0}^{-1}n+n^{1/3}<k\leq\alpha^{-1}n}\abs{\widehat{B}(k)}\lesssim\log n,
\]
and for $p\neq1$
\begin{align*}
\sum_{\alpha_{0}^{-1}n+n^{1/3}<k\leq\alpha^{-1}n}\abs{\widehat{B}(k)}^{p}\lesssim\frac{1}{n^{p-1}}\frac{n^{\frac{2p-2}{3}}}{p-1}\lesssim\frac{1}{p-1}\frac{1}{n^{\frac{p-1}{3}}}.
\end{align*}

\item We estimate
\begin{align*}
\sum_{\alpha_0^{-1}n-n^{1/3}<k\leq\alpha_0^{-1}n+n^{1/3}}\abs{\widehat{B}(k)}^{p} & =\sum_{\abs{k/n-\alpha_{0}^{-1}}<\frac{1}{n^{2/3}}}\abs{\widehat{B}(k)}^{p} \leq\frac{c^{p}}{n^{p/3}}\sum_{\abs{k/n-\alpha_{0}^{-1}}<\frac{1}{n^{2/3}}}1\lesssim\frac{1}{n^{\frac{p-1}{3}}}.
\end{align*}
\item
We estimate
\begin{align*}
\sum_{\beta n<k\leq \alpha_0^{-1}n-n^{1/3}}\abs{\widehat{B}(k)}^{p}\leq\frac{c^{p}}{n^{p/2-1}}\frac{1}{n}\sum_{{\beta n<k\leq \alpha_0^{-1}n-n^{1/3}}}\frac{1}{\left(\alpha_{0}^{-1}-\frac{k}{n}\right)^{p/4}}.
\end{align*}
We set $g(t)=\frac{1}{\left(\alpha_{0}^{-1}-t\right)^{p/4}}$
and we bound the Riemann sum
\begin{align*}
\frac{1}{n}\sum_{{\beta<k/n\leq\alpha_{0}^{-1}-n^{-2/3}}}g\left(\frac{k}{n}\right) & \lesssim\int_{\beta}^{\alpha_{0}^{-1}-n^{-2/3}}g(t){\rm d}t\\
 & {\lesssim\int_{n^{-2/3}}^{\alpha_{0}^{-1}-\beta}\frac{1}{u^{p/4}}{\rm d}u},
\end{align*}
{where $\alpha_{0}^{-1}-\beta=\alpha_{0}^{-1}-\frac{\alpha_{0}+\alpha_{0}^{-1}}{2}=\frac{\alpha_{0}^{-1}-\alpha_{0}}{2}>0$.}
Evaluating the integral gives the following upper estimates. If $p=4$
then 
\[
\sum_{\beta n<k\leq\alpha_{0}^{-1}n-n^{1/3}}\abs{\widehat{B}(k)}^{4}\lesssim\frac{\log n}{n},
\]
if $p>4$ then 
\begin{align*}
\sum_{\beta n<k\leq\alpha_{0}^{-1}n-n^{1/3}}\abs{\widehat{B}(k)}^{p} & \lesssim\frac{1}{p-4}\frac{n^{\frac{p-4}{6}}}{n^{p/2-1}}\lesssim\frac{1}{p-4}\frac{1}{n^{\frac{p-1}{3}}}
\end{align*}
and if $p<4$ then 
\[
\sum_{\beta n<k\leq\alpha_{0}^{-1}n-n^{1/3}}\abs{\widehat{B}(k)}^{p}\lesssim\frac{1}{4-p}\frac{1}{n^{p/2-1}}.
\]
\end{itemize}
\end{proof}

\section{\label{sec:Lower-bounds}Lower Estimates}

Before going into the details of a technical discussion of sharpness in Theorem~\ref{main} we summarize some known facts and provide a simple argument for sharpness in the interval $p\in(4/3,4)$. The discussion of sharpness in the interval $p\in(4,\infty)$ is build on an expansion of the Taylor coefficients of $B$ in terms of the Airy function.
The most complicated case turns out the boundary case $p=4$, which is treated separately in the end. We notice that the upper bound in Theorem \ref{main} is sharp for $p=\infty$ as a consequence of \cite[Theorem 2, point (2)]{SZ}. In that reference the behavior of the largest coefficient, $\sup_{k\geq0}\abs{\hat{B}(k)}$ is analyzed and it is shown that to first order we have $\sup_{k\geq0}\Abs{\hat{B}(k)}\sim n^{-1/3}$.
For the case $p=1$ the limit $\lim_{n\rightarrow\infty}\Norm{B}{l_1^A}/\sqrt{n}$ is computed in \cite{JK,DG,BHl}. The case $p=2$ is trivial and a direct consequence of Plancharel's theorem
\begin{align*}
\Norm{B}{l_2^A}=\sum_{k\geq0}\abs{\widehat{B}(k)}^2=<B,B>_{L_2(\partial\mathbb{D})}=1,
\end{align*}
where the last step makes use of the fact that $\bar{B}(z)=1/B(z)$ for $z\in\partial\mathbb{D}$.
Lower estimates are derived in the whole range $p\in[1,2]$ in \cite{BS}, which imply that the theorem is sharp in this region. The asymptotic statement of Theorem~\ref{main} is actually sharp for $p\in[1,\,4)$. This is a consequence of the upper estimate in Theorem~\ref{main}.
\begin{cor}\label{tral}
For $p\in[1,4)$ we have
\[
\Norm{B}{l_{p}^{A}}\asymp c_p n^{\frac{2-p}{2p}}.
\]
\end{cor}%
Here is an elementary proof which holds curiously
for $p\in(\frac{4}{3},\,4)$ only. (The rest of the interval is covered already in~\cite{BS} and there is no need to reproduce the discussion.)
\begin{proof}[Proof of corollary~\ref{tral}]
Let $p'$ denote the  Hölder conjugate of $p$ i.e.~$\frac{1}{p}+\frac{1}{p'}=1$. By assumption $\frac{1}{p'}>1-\frac{3}{4}$
and $p'<4$. It follows from the upper estimates in Theorem~\ref{main} (proved in the previous section) that
\begin{align*}
\Norm{b_{\lambda}^{n}}{l_{p'}^A} & \leq C_{p'}n^{\frac{2-p'}{2p'}} =C_{p'}n^{\frac{p-2}{2p}}.
\end{align*}
A straightforward application of Hölder's inequality gives: 
\[
1=\Norm{b_{\lambda}^{n}}{l_{2}^A}^{2}\leq\Norm{b_{\lambda}^{n}}{l_{p}^A}\Norm{b_{\lambda}^{n}}{l_{p'}^A}.
\]
We conclude that 
\begin{align*}
\Norm{b_{\lambda}^{n}}{l_{p}^A} & \geq\frac{1}{\Norm{b_{\lambda}^{n}}{l_{p'}^A}}\geq\tilde{C}_{p}n^{\frac{2-p}{2p}}
\end{align*}
with $\tilde{C}_{p}=\frac{1}{C_{p'}}$. 
\end{proof}%
For our discussion of sharpness we can henceforth assume that $p\in[4,\infty]$. We have already seen in Proposition~\ref{main_Lemma} point (1) that if $k=k(n)$ is a sequence with $k/n\leq\alpha$ of $k/n\geq\alpha$ for some $\alpha\in\left(0,\alpha_0=\frac{1-\lambda}{1+\lambda}\right)$ then $\abs{\hat{B}(k)}$ decay exponentially as $n\rightarrow\infty$. This means that for any $p\in[1,\infty]$ the main contribution in the $l_{p}-$norms of $\widehat{B}$ is due to a critical range of $k$ with $k\in[\alpha_0n,\,\alpha_0^{-1}n]$. In this section we compute an asymptotic expansion of $\hat{B}(k)$ as $k$ and $n$ tend simultaneously to $\infty$ and $k$ approaches the right boundary of $[\alpha_0n,\,\alpha_0^{-1}n]$ from inside: 
\[
\lim_{n\rightarrow\infty}\left(\alpha_0^{-1}-\frac{k}{n}\right)=0^{+}.
\]
In this region the asymptotic behavior of $\hat{B}(k)$ can be written in terms of the Airy function $Ai(x)$. For real arguments the latter can be defined as an improper Riemann integral
$$Ai(x)=\frac{1}{\pi}\int_0^\infty\cos\left(\frac{t^3}{3}+xt\right)\d t.$$
For us the most interesting will be the oscillatory behavior of $Ai$ for large negative arguments. We have the asymptotic approximation
\[
Ai(-x)\sim\frac{1}{x^{1/4}\sqrt{\pi}}\cos\left(\frac{2}{3}x^{3/2}-\frac{\pi}{4}\right),\qquad x\rightarrow+\infty.
\]

\begin{prop}[Asymptotic expansion of $\hat{B}(k)$ for $k$ in a left neighborhood of $\alpha_0^{-1}n$]
\label{thm:main_Th_1}
Let $B=b_{\lambda}^{n}$ with $\lambda\in(0,1)$
and $n\geq1$ and set $\alpha_{0}:=\frac{1-\lambda}{1+\lambda}$. Consider sequences $k=k(n)$ with $k\in[\alpha_0n,\alpha_0^{-1}n]$
such that $\lim_{n\rightarrow\infty}\frac{k}{n}=\alpha_0^{-1}$. 
For the Taylor coefficients of $B$ we have the following asymptotic expansion as $n\rightarrow\infty$
\[
\widehat{B}(k)=\frac{(1-\lambda)^{1/4}}{\left(\lambda(1+\lambda)\right)^{1/12}}\frac{\sqrt{2}}{\sqrt{\frac{k}{n}}\left(\frac{k}{n}-\alpha_0\right)^{1/4}}\frac{Ai(-n^{2/3}\delta^{2})}{n^{1/3}}\left(1+\cO(n^{-1/3})\right),
\]
where $\delta^{2}=\frac{1-\lambda}{\left(\lambda(1+\lambda)\right)^{1/3}}\left(\alpha_0^{-1}-\frac{k}{n}\right).$

\end{prop}
The slowest decay of $\widehat{B}(k)$ occurs at the boundary $k=\alpha_0^{-1}n$. Here the supremum $\abs{\widehat{B}(k)}$ is attained and corresponds to the $l_\infty^A$-norm. In this situation we can recover some of the findings of \cite[Proposition 4]{SZ}. We find that the boundary behavior as $n$ gets large (at $k=\alpha_0^{-1}n$) is
\begin{align*}
\widehat{B}(k) & \sim\frac{(1-\lambda)^{1/4}}{\left(\lambda(1+\lambda)\right)^{1/12}}\frac{\sqrt{2}}{\sqrt{\frac{1+\lambda}{1-\lambda}}\left(\frac{1+\lambda}{1-\lambda}-\frac{1-\lambda}{1+\lambda}\right)^{1/4}}\frac{Ai(0)}{n^{1/3}}=\frac{1-\lambda}{\left(\lambda(1+\lambda)\right)^{1/3}}\frac{1}{3^{2/3}\Gamma(2/3)}\frac{1}{n^{1/3}},
\end{align*}
which proves (as already shown in~\cite{SZ}) that $\Norm{B}{l_\infty^A}\asymp n^{-1/3}$. This can be extended to the whole interval $p\in(4,\infty]$. We find the corresponding corollary to Proposition~\ref{thm:main_Th_1} and Proposition~\ref{main_Lemma}.
\begin{cor}\label{muda}
For $p\in(4,+\infty]$
\[
\Norm{B}{l_{p}^{A}}\asymp n^{\frac{1-p}{3p}}.
\]
\end{cor}
\begin{proof}[Proof of Corollary~\ref{muda}]
We consider $p\in(4,\infty)$ as the case $p=\infty$ is clear from above. Since $
\lim_{n\rightarrow\infty}\left(\alpha_0^{-1}-\frac{k}{n}\right)=0$
the prefactor in Proposition~\ref{thm:main_Th_1} is comparable to a
positive constant
\[
\frac{(1-\lambda)^{1/4}}{\left(\lambda(1+\lambda)\right)^{1/12}}\frac{\sqrt{2}}{\sqrt{\frac{k}{n}}\left(\frac{k}{n}-\frac{1-\lambda}{1+\lambda}\right)^{1/4}}\asymp1.
\]
We consider the set $I_{n}$ of integers in $\left[\alpha_0^{-1}n-cn^{1/3},\,\alpha_0^{-1}n\right]$
\[
I_{n}=\mathbb{N}\cap\left[\alpha_0^{-1}n-cn^{1/3},\,\alpha_0^{-1}n\right],
\]
where the constant $c>0$ is chosen such that $n^{2/3}\delta^{2}<2$ for $k\in I_{n}$. Explicitly this condition reads as
\[
n^{2/3}\delta^2=\frac{1-\lambda}{\left(\lambda(1+\lambda)\right)^{1/3}}\left(\alpha_0^{-1}-\frac{k}{n}\right)n^{2/3}\leq c\frac{1-\lambda}{\left(\lambda(1+\lambda)\right)^{1/3}}< 2
\]
i.e.~we can choose $c<2\frac{\left(\lambda(1+\lambda)\right)^{1/3}}{1-\lambda}$. This choice of $c$ ensures that for $k\in I_{n}$ the quantity $-n^{2/3}\delta^{2}$ lies
in the compact interval $\left[-2,0\right]$ on which the Airy function takes values that are separated from $0$,
\[
\abs{Ai(-n^{2/3}\delta^{2})}\geq\min_{\xi\in\left[-2,0\right]}\abs{Ai(\xi)}>0.
\]
(The first negative zero of the Airy function occurs at approximately $-2.33811$).
In other words from Proposition~\ref{thm:main_Th_1} we have for sufficiently large $n$ and all $k\in I_n$ an estimate of the form
\[
\abs{\hat{B}(k)}\geq\frac{K_{1}}{n^{1/3}}\left(1+\cO\left({n^{-1/3}}\right)\right)
\]
with a constant $K_1>0$. Therefore we have
\[
\Norm{B}{l_{p}^{A}}^{p}\gtrsim\frac{\#I_{n}}{n^{p/3}}\gtrsim\frac{n^{1/3}}{n^{p/3}}.
\]
\end{proof}
Notice that for $p<4$ we have that ${\frac{2-p}{2p}}>\frac{1-p}{3p}$ such that the argument given above does not reach the lower bound of Corollary~\ref{tral} for the interval $p\in[1,4)$. Since Proposition~\ref{thm:main_Th_1} provides the exact asymptotic behavior we can conclude that the dominant contribution to the $l_p$ norms of $\widehat{B}$ does not come from the interval $I_n$ when $p\in[1,4)$. Instead for $p>4$ estimates are achieved in the region \emph{IV} of Table~\ref{PussWideOpen}. As we will see in this situation the main contribution to
\begin{align*}
\hat{B}(k)=\overline{\hat{B}(k)} =\frac{1}{2\pi}\int_{-\pi}^{\pi}b_{\lambda}^{-n}(e^{i\varphi})e^{ik\varphi}\d\varphi.
\end{align*}
comes from a small interval around $\varphi=0$, see \cite[Proposition 4, point 2)]{SZ}. For technical convenience we focus our analysis on the integral representation of $\overline{\hat{B}(k)}$, which is the same as $\hat{B}(k)$ as $\lambda$ is real. We fix $\varepsilon\in(0,\pi)$ and split $\hat{B}(k)$ as
\[
\hat{B}(k)=\frac{1}{2\pi}\int_{-\varepsilon}^{\varepsilon}b_{\lambda}^{-n}(z)z^{k}\Bigg|_{z=e^{i\varphi}}\d\varphi+\frac{1}{2i\pi}\int_{\partial\mathbb{D}\setminus\mathcal{C}_{\varepsilon}}b_{\lambda}^{-n}(z)z^{k-1}\d z,
\]
where $\cC_\varepsilon=\{z=e^{i\varphi}\:|\:\varphi\in(-\varepsilon,\varepsilon)\}$. We write the integrals in a way that is convenient for asymptotic analysis.
%
%
%
%
%
%
We introduce a function $h_a$ with $a\in\mathbb{R}^+$ and
\begin{align*}
h_a(z) & =-i\log{\left(\frac{z^{a}(1-\lambda z)}{z-\lambda}\right)}=\textnormal{Arg}\left(\frac{z^{a}(1-\lambda z)}{z-\lambda}\right),
\end{align*}
where $\log$ denotes the principal branch of the complex logarithm. We
have 
\begin{align*}
\frac{1}{2i\pi}\int_{\partial\mathbb{D}\setminus\mathcal{C}_{\varepsilon}}b_{\lambda}^{-n}(z)z^{k-1}\d z & =\frac{1}{\pi}\Re\left\{ \int_{\varepsilon}^{\pi}{z^{k}\bigg(\frac{1-\lambda z}{z-\lambda}\bigg)^{n}}\Bigg|_{z=e^{i\varphi}}\d\varphi\right\} \\
 & =\frac{1}{\pi}\Re\left\{ \int_{\varepsilon}^{\pi}e^{inh_{k/n}(z)}\Bigg|_{z=e^{i\varphi}}\d\varphi\right\}.
\end{align*}
To prove Theorem \ref{thm:main_Th_1} we proceed by the following steps
\begin{enumerate}
\item We prove that $\int_{\partial\mathbb{D}\setminus\mathcal{C}_{\varepsilon}}b_{\lambda}^{n}(z)z^{-k-1}\d z=\cO\left(\frac{1}{n}\right)$, where we make use of a Van der Corput type lemma, see Lemma \ref{lem:VDCP_lemma_delta_fixed} below.

\item We compute an asymptotic expansion for $\frac{1}{2\pi}\int_{-\varepsilon}^{\varepsilon}b_{\lambda}^{-n}(z)z^{k}\Bigg|_{z=e^{i\varphi}}\d\varphi$ relying on the so-called~\textit{uniform method of stationary phase}~\cite[Section 2.3 p. 41]{VB}. The technical core of the latter will be a locally one-to-one cubic transformation of the integrand's argument following the methods of \cite{CFU}. 
\end{enumerate}
\begin{lem}
\label{lem:VDCP_lemma_delta_fixed} Let $k(n)$ be a sequence that approaches $\alpha_0^{-1}n$ from the left, i.e.~$\alpha_0^{-1}-\frac{k}{n}\rightarrow0^+$. Given fixed $\varepsilon\in(0,\pi)$ we have as $n\rightarrow\infty$ that
\[
\hat{B}(k)=\frac{1}{2\pi}\int_{-\varepsilon}^{\varepsilon}e^{inh_{k/n}\left(z\right)}\Bigg|_{z=e^{i\varphi}}\d\varphi+\cO\left(\frac{1}{n}\right).
\]
 \end{lem}

\begin{proof} We define a function $\tilde{h}_a$ on $(-\pi,\pi)$ by setting $\tilde{h}_a(\varphi):=h_a(e^{i\varphi})$. Deriving $\tilde{h}$
with respect to $\varphi$ we find 
\[
\frac{\partial\tilde{h}_{k/n}}{\partial\varphi}(\varphi)=\frac{k}{n}-\frac{1-\lambda^{2}}{1+\lambda^{2}-2\lambda\cos\varphi}.
\]
For $\varphi\in[\varepsilon,\pi)$ we have 
\begin{align*}
n\frac{\partial\tilde{h}_{k/n}}{\partial\varphi}(\varphi) & \geq n\left(\frac{k}{n}-\frac{1-\lambda^{2}}{1+\lambda^{2}-2\lambda(1-\varepsilon)}\right)\\
 & =n\left(\alpha_{0}^{-1}-\frac{1-\lambda^{2}}{1+\lambda^{2}-2\lambda(1-\varepsilon)}-\epsilon_{n}\right)\\
 & =n\left(\frac{2\lambda\varepsilon(1+\lambda)}{\left(1+\lambda^{2}-2\lambda(1-\varepsilon)\right)(1-\lambda)}-\epsilon_{n}\right),
\end{align*}
where 
\[
\epsilon_{n}:=\alpha_{0}^{-1}-\frac{k}{n}\rightarrow0^{+}
\]
as $n$ tends to $\infty$. For $n$ sufficiently large
\[
n\frac{\partial\tilde{h}_{k/n}}{\partial\varphi}(\varphi)\geq constant*n
\]
for any $\varphi\in[\varepsilon,\pi)$. Therefore an application of Lemma~\ref{VDCP} provides
\[
\int_{\varepsilon}^{\pi}e^{inh_{k/n}(z)}\Bigg|_{z=e^{i\varphi}}\d\varphi=\cO\left(\frac{1}{n}\right).
\]

\end{proof}
The next step is to compute an asymptotic expansion of
\[
J_{n}(k)=\frac{1}{2\pi}\int_{-\varepsilon}^{\varepsilon}e^{inh_{k/n}\left(z\right)}\Bigg|_{z=e^{i\varphi}}\d\varphi.
\]
We will see that $J_{n}(k)$ is well suited for
an application of the uniform method of stationary phase \cite[Section 2.3 p. 41]{VB}, which is based in turns on a locally one-to-one cubic transformation of
$h$, which is described in~\cite{CFU} and \cite[p.~366]{RW},~\cite[p.~369]{BH}).
In order to apply the result from \cite{CFU} we perform a locally one-to-one cubit transformation to the function $h$. First we notice that both $\alpha_{0}$ and $\alpha_{0}^{-1}$ are critical values in the sense that for $\alpha\notin\{\alpha_{0},\alpha_{0}^{-1}\}$ the function $h$ has two distinct saddle points $z_{+}$
and $z_{-}$ of rank $1$. However, if $\alpha\in\{\alpha_{0},\alpha_{0}^{-1}\}$ the points $z_+$ and $z_-$ merge to a single saddle point $z_{0}$ (respectively $\tilde{z}_{0}$) of rank $2$. For notational convenience we shall write the function $h_\alpha$ with an additional argument instead of the index, $h_\alpha(z)=h(z,\alpha)$. To be precise the conditions for mentioned saddle points read
\begin{equation}
\frac{\partial h}{\partial z}(z_{+},\alpha)=\frac{\partial h}{\partial z}(z_{-},\alpha)=0,\qquad\frac{\partial^{2}h}{\partial z^{2}}(z_{\pm},\alpha)\neq0\label{eq:2d_assump_wong}
\end{equation}
for $\alpha\notin\{\alpha_{0},\alpha_{0}^{-1}\}$ for saddle-points $z_+$, $z_-$ of rank one and
\[
\frac{\partial h}{\partial z}(z_{0},\alpha_{0})=\frac{\partial^{2}h}{\partial z^{2}}(z_{0},\alpha_{0})=0,\qquad\frac{\partial^{3}h}{\partial z^{3}}(z_{0},\alpha_{0})\neq0,
\]
respectively
\begin{equation}
\frac{\partial h}{\partial z}(\tilde{z}_{0},\alpha_{0}^{-1})=\frac{\partial^{2}h}{\partial z^{2}}(\tilde{z}_{0},\alpha_{0}^{-1})=0,\qquad\frac{\partial^{3}h}{\partial z^{3}}(\tilde{z}_{0},\alpha_{0}^{-1})\neq0\label{eq:1st_assumpt_wong_tilde}
\end{equation}
for saddle points of rank $2$. Computing derivatives we find 
\begin{align*}
i\frac{\partial h}{\partial z} & =-\frac{1}{z-\lambda}+\frac{\alpha}{z}-\frac{\lambda}{1-\lambda z},\\
i\frac{\partial^{2}h}{\partial z^{2}} & =\frac{1}{(z-\lambda)^{2}}-\frac{\alpha}{z^{2}}-\frac{\lambda^{2}}{(1-\lambda z)^{2}},\\
i\frac{\partial^{3}h}{\partial z^{3}} & =-\frac{2}{(z-\lambda)^{3}}+\frac{2\alpha}{z^{3}}-\frac{2\lambda^{3}}{(1-\lambda z)^{3}}.
\end{align*}
The function $h(z,\alpha)$ has a stationary point if and only if
$\partial f/\partial z=0$, i.e.~iff 
\begin{align*}
\alpha=1+\frac{\lambda}{z-\lambda}+\frac{\lambda z}{1-\lambda z}.
\end{align*}
Solving the latter for $z$ gives 
\[
z_{\pm}=\frac{\alpha(1+\lambda^{2})-(1-\lambda^{2})}{2\lambda\alpha}\pm i\sqrt{1-\left(\frac{\alpha(1+\lambda^{2})-(1-\lambda^{2})}{2\lambda\alpha}\right)^{2}}\in\partial\mathbb{D}
\]
and we write $z_{\pm}=e^{i\varphi_{\pm}}$ with $\varphi_{+}\in[0,\pi]$
and $\varphi_{-}\in(-\pi,0]$. Observe that $\varphi_{+}=\varphi_{+}(\alpha)$,
$\varphi_{-}=-\varphi_{+}$and 
\[
\cos^{2}\varphi_{+}=\frac{\alpha(1+\lambda^{2})-(1-\lambda^{2})}{2\lambda\alpha}.
\]

We distinguish the two cases \emph{1)} $\alpha\in(\alpha_{0},\,\alpha_{0}^{-1})$
and \emph{2)} $\alpha\in\left\{ \alpha_{0},\,\alpha_{0}^{-1}\right\} $,
which are characterized by the presence of a stationary point of order
one ($\frac{\partial h}{\partial z}(z_{\pm},\alpha)=0$ but $\frac{\partial^{2}h}{\partial z^{2}}(z_{\pm},\alpha)\neq0$)
in \emph{Case 1)} and of order two ($\frac{\partial h}{\partial z}(z_{\pm},\alpha)=\frac{\partial^{2}h}{\partial z^{2}}(z_{\pm},\alpha)=0$
but $\frac{\partial^{3}h}{\partial z^{3}}(z_{\pm},\alpha)\neq0$)
in \emph{Case 2)}.

\emph{Case~1)} If $\alpha\in(\alpha_{0},\,\alpha_{0}^{-1})$ then
the zeros $z_{+}=e^{i\varphi_{+}}$ and $z_{-}=e^{i\varphi_{-}}$
of $\frac{\partial f}{\partial z}$ are distinct points located on
$\partial\mathbb{D}$ with $\varphi_{+}\in[0,\pi]$ and $\varphi_{-}\in(-\pi,0]$.
Plugging in we see that 
\begin{eqnarray}
i\frac{\partial^{2}h}{\partial z^{2}}\Bigg|_{z=z_{\pm}} & = & \frac{(1-\lambda^{2})(1-z_{\pm}^{2})\lambda}{z_{\pm}(z_{\pm}-\lambda)^{2}(1-\lambda z_{\pm})^{2}}.\label{eq:scd_deriv}
\end{eqnarray}

\emph{Case 2) } If $\alpha\in\left\{ \alpha_{0},\,\alpha_{0}^{-1}\right\} $
then $\frac{\partial h}{\partial z}$ has a unique zero. If $\alpha=\alpha_{0}^{-1}$
then $z_{+}=z_{-}=1=\tilde{z}_{0}$ and 
\[
f(1,\alpha_{0}^{-1})=\frac{\partial f}{\partial z}(1,\alpha_{0}^{-1})=\frac{\partial^{2}f}{\partial z^{2}}(1,\alpha_{0}^{-1})=0,
\]
with 
\[
i\frac{\partial^{3}h}{\partial z^{3}}(1,\alpha_{0}^{-1})=-\frac{2\lambda(1+\lambda)}{(1-\lambda)^{3}}\neq0.
\]
If $\alpha=\alpha_{0}$ then $z_{+}=z_{-}=-1=z_{0}$ and 
\[
i\frac{\partial h}{\partial z}(-1,\alpha_{0})=i\frac{\partial^{2}h}{\partial z^{2}}(-1,\alpha_{0})=0,\qquad i\frac{\partial^{3}h}{\partial z^{3}}(-1,\alpha_{0})=-\frac{2\lambda(1-\lambda)}{(1+\lambda)^{3}}\neq0.
\]
The contour of integration $\mathcal{C}_{\varepsilon}$ in $J_{n}(k)$ is chosen so that it is located in a neighborhood of $\tilde{z}_{0}=1$.
This leads to considering the right boundary $\alpha_{0}^{-1}$ so that $z_{+}$ and $z_{-}$ lie in $\mathcal{C}_{\varepsilon}$ for some $\varepsilon>0$ and $\alpha$ close to $\alpha_{0}^{-1}.$ \eqref{eq:1st_assumpt_wong_tilde}
and \eqref{eq:2d_assump_wong} are clearly satisfied. We are now ready
to perform the one-to-one cubic transformation of~\cite{CFU}. 
\begin{prop} \label{thm:Chester} For $\alpha$ near $\alpha_{0}^{-1}$
the cubic transformation

\[
h(z,\alpha)=\frac{t^{3}}{3}-\frac{(\alpha_{0}^{-1}-\alpha)(1-\lambda)}{\left(\lambda(1+\lambda)\right)^{1/3}}t
\]
has exactly one branch $t=t(z,\alpha)$ which can be expanded into
a power series in $z,$ with coefficients which are continuous in
$\alpha$. On this branch the points $z=z_{\pm}$ correspond, respectively,  to $t=\pm\gamma$. Furthermore the mapping of $t$ to $z$ is one-to-one locally on a neighborhood of $0$ onto $\mathcal{C}_{\varepsilon}$
for some positive $\varepsilon$. \end{prop}
\begin{proof} Following \cite{CFU} let us define $z(t)$ by the
equation 
\begin{equation}
h(z,\alpha)=\frac{t^{3}}{3}-\gamma^{2}t+\rho\label{eq:change_var}
\end{equation}
where $\gamma=\gamma(\alpha)$ and $\rho=\rho(\alpha)$ are to be
determined. This transformation is shown in \cite{CFU} to be locally
one-to-one and analytic for all $\alpha$ in a neighborhood of $\alpha_{0}^{-1}$.
By differentiating the above equation with respect to $t$ we obtain
\begin{equation}
z'(t)=\frac{t^{2}-\gamma^{2}}{\frac{\partial h}{\partial z}\left(z(t),\alpha\right)}.\label{eq:diff1}
\end{equation}
Since $t\mapsto z(t)$ should yield a conformal map we must require
that $z'$ is finite and nonzero. We see from the above equality that
difficulties can only arise when $z=z_{\pm}$ and when $t=\pm\gamma$.
We change variables such that we have $t=\pm\gamma$ when $z=z_{\pm}$.
More precisely it follows from \cite{CFU} that (see \cite[Theorem 1 p.368]{RW}):

1) the parameters $\gamma=\gamma(\alpha)$ and $\rho=\rho(\alpha)$
can be explicitely determined so that the transformation \eqref{eq:change_var}
has exactly one branch $t=t(z,\alpha)$ which can be expanded into
a power series in $z,$ with coefficients which are continuous in
$\alpha$ for $\alpha$ near $\alpha_{0}^{-1}$,

2) on this branch the points $z=z_{\pm}$ correspond to $t=\pm\gamma$
respectively,

3) for $\alpha$ near $\alpha_{0}^{-1}$ the correspondence $t$ to $z$
is locally one-to-one that is to say from a neighborhood of 0 onto
a neighborhood of $1$ say $\mathcal{C}_{\varepsilon}$ for some positive
$\delta$. Indeed the two simple saddle points $z_{+}=z(t_{+})$ and
$z_{-}=z(t_{-})$ for $z\mapsto f(z,\alpha)$ (resp. $t_{+}=\gamma$
and $t_{-}=-\gamma$ for $t\mapsto-\frac{t^{3}}{3}+\gamma^{2}t+\rho$)
coalesce to a single saddle point of order $2$ when $z_{+}=z_{-}=1=z_{0}$
or equivalentely when $t_{+}=t_{-}=\gamma=0=:t_{0}$.
\textbf{Determination of $\gamma$ and $\rho$}.We show that 
\[
\gamma=\frac{\sqrt{(\alpha_{0}^{-1}-\alpha)(1-\lambda)}}{\left(\lambda(1+\lambda)\right)^{1/6}},\qquad\rho=0.
\]
It follows from \eqref{eq:change_var} that 
\[
h(z(t_{+}),\alpha)=h(z(\gamma),\alpha)=-\frac{2\gamma^{3}}{3}+\rho
\]
and 
\[
h(z(t_{-}),\alpha)=h(z(-\gamma),\alpha)=\frac{2\gamma^{3}}{3}+\rho
\]
so that

\[
\gamma^{3}=\frac{3}{4}\left(f(z_{-},\alpha)-f(z_{+},\alpha)\right)
\]
and 
\[
\rho=\frac{1}{2}\left(f(z_{+},\alpha)+f(z_{-},\alpha)\right).
\]
Since $z_{+}=\overline{z_{-}}$ 
\[
\frac{z_{+}^{\alpha}(1-\lambda z_{+})}{z_{+}-\lambda}=e^{if(z_{+},\alpha)}
\]
implies 
\[
\frac{z_{-}^{\alpha}(1-\lambda z_{-})}{z_{-}-\lambda}=e^{-if(z_{+},\alpha)},
\]
which gives us 
\[
\gamma^{3}=-\frac{3}{2}f(z_{+},\alpha),\qquad\rho=0.
\]
We observe that $\gamma$ is not uniquely determined by the above
equality. Indeed when $z_{+}\neq z_{-}$ it defines three values of
$\gamma$. We discuss below this ambiguity and compute $\arg\gamma$.

{Computation of $\arg\gamma$.}

1) Plugging $t=0$ in \eqref{eq:diff1} with $\alpha$ close to $\alpha_{0}^{-1}$
we get 
\begin{equation}
z'(0)=\frac{-\gamma^{2}}{\frac{\partial f}{\partial z}\left(1,\alpha\right)}=\frac{i\gamma^{2}}{\alpha_{0}^{-1}-\alpha}\label{eq:simple}
\end{equation}
which yields (since $\alpha<\alpha_{0}^{-1}$) 
\begin{equation}
\arg z'(0)=\frac{\pi}{2}+2\arg\gamma\qquad\mod2\pi.\label{eq:1st_equation_zprime0}
\end{equation}

2) We observe that $\gamma=\gamma(\alpha)$ and $\lim_{\alpha\rightarrow\alpha_{0}^{-1}}\gamma=0$
so that 
\[
z'(0)=\lim_{\gamma\rightarrow0}\frac{z(t_{+})-z(0)}{t_{+}}=\lim_{\alpha\rightarrow\alpha_{0}^{-1}}\frac{z_{+}-1}{\gamma}
\]
and 
\[
z'(0)=\lim_{\gamma\rightarrow0}\frac{z(t_{-})-z(0)}{t_{-}}=\lim_{\alpha\rightarrow\alpha_{0}^{-1}}\frac{z_{-}-1}{-\gamma}
\]
which gives 
\[
z'(0)=\lim_{\alpha\rightarrow\alpha_{0}^{-1}}\frac{z_{+}-z_{-}}{2\gamma}.
\]
Since $z_{+}-z_{-}\in i\mathbb{R}_{+}$ (whatever $\alpha$ is close
to $\alpha_{0}^{-1}$ or not) the above equality implies that 
\begin{equation}
\arg z'(0)=\frac{\pi}{2}-\arg\gamma\qquad\mod2\pi.\label{eq:2d_equ_zprime0}
\end{equation}
It follows from \eqref{eq:1st_equation_zprime0} and \eqref{eq:2d_equ_zprime0}
that $\gamma^{3}>0$.

3) We use finally $z'(t_{\pm})$. Differentiating \eqref{eq:diff1}
with respect to $t$ and specifying the corresponding identity at
$t=t_{\pm}=\pm\gamma$ we get 
\[
\left(z'(t_{\pm})\right)^{2}=\frac{2t_{\pm}}{\frac{\partial^{2}h}{\partial z^{2}}\left(z_{\pm},\alpha\right)}
\]
that is to say 
\[
\left(z'(t_{+})\right)^{2}=\frac{2\gamma}{\frac{\partial^{2}h}{\partial z^{2}}\left(z_{+},\alpha\right)},\qquad\left(z'(t_{-})\right)^{2}=-\frac{2\gamma}{\frac{\partial^{2}h}{\partial z^{2}}\left(z_{-},\alpha\right)}.
\]
It follows from \eqref{eq:scd_deriv} that 
\begin{equation}
\frac{\partial^{2}h}{\partial z^{2}}\Bigg|_{z=z_{+}}=-\frac{i(1-\lambda^{2})(z_{-}-z_{+})\lambda}{(z_{+}-\lambda)^{2}(1-\lambda z_{+})^{2}}=-z_{+}^{-2}\frac{2\Im(z_{+})(1-\lambda^{2})\lambda}{\abs{1-\lambda z_{+}}^{4}}\label{eq:scd_deriv_f_z_plus}
\end{equation}
and 
\begin{equation}
\frac{\partial^{2}h}{\partial z^{2}}\Bigg|_{z=z_{-}}=-\frac{i(1-\lambda^{2})(z_{+}-z_{-})\lambda}{(z_{-}-\lambda)^{2}(1-\lambda z_{-})^{2}}=z_{-}^{-2}\frac{2\Im(z_{+})(1-\lambda^{2})\lambda}{\abs{1-\lambda z_{+}}^{4}}.\label{eq:scd_derivat_f_z_minus}
\end{equation}
In particular taking the arguments 
\[
\arg\left(z'(t_{\pm})\right)^{2}=\pi+\arg\gamma+2\varphi_{\pm}\qquad\mod2\pi
\]
where $\varphi_{\pm}=\varphi_{\pm}(\alpha)\rightarrow0$ as $\alpha$
tends to $\alpha_{0}^{-1}.$ Passing after to the limit as $\alpha\rightarrow\alpha_{0}^{-1}$
we obtain the third equation 
\begin{equation}
\arg\left(z'(0)\right)^{2}=\pi\qquad\mod2\pi\label{eq:3_rd_equation_z_prime0}
\end{equation}
because $\gamma=\gamma(\alpha)\rightarrow0$ as $\alpha\rightarrow\alpha_{0}^{-1}.$
Adding \eqref{eq:1st_equation_zprime0} and \eqref{eq:2d_equ_zprime0}
and comparing it with \ref{eq:3_rd_equation_z_prime0} we find 
\[
\pi+\arg\gamma=2\arg\left(z'(0)\right)=\pi\qquad\mod2\pi
\]
and we conclude that $\gamma>0$ and $\arg\, z'(0)=\frac{\pi}{2}$.

{Computation of $\gamma$ and $z'(0)$.} Differentiating \eqref{eq:diff1}
two times with respect to $t$ and specifying the corresponding identity
at $\alpha=\alpha_{0}^{-1}$ and $t=t_{0}$ we get 
\begin{equation}
\left(z'(0)\right)^{3}=\frac{2}{\frac{\partial^{3}h}{\partial z^{3}}(1,\alpha_{0}^{-1})}=\frac{(1-\lambda)^{3}}{i\lambda(1+\lambda)}.\label{eq:z_prime_0_sd}
\end{equation}
Since $\arg z'(0)=\frac{\pi}{2}$ 
\[
z'(0)=i\frac{1-\lambda}{\left(\lambda(1+\lambda)\right)^{1/3}}.
\]
Together with \eqref{eq:simple} this gives 
\[
\gamma^{2}=\frac{(\alpha_{0}^{-1}-\alpha)(1-\lambda)}{\left(\lambda(1+\lambda)\right)^{1/3}},\qquad\gamma=\frac{\sqrt{(\alpha_{0}^{-1}-\alpha)(1-\lambda)}}{\left(\lambda(1+\lambda)\right)^{1/6}}.
\]
\end{proof}
With the developed theory we are ready to conclude the proof of Proposition~\ref{thm:main_Th_1}. We will apply the uniform method of stationary phase~\cite[Section 2.3 pp. 41-44]{VB}
to 
\[
J_{n}(k)=\frac{1}{2\pi}\int_{-\varepsilon}^{\varepsilon}e^{inh\left(z,\frac{k}{n}\right)}\Bigg|_{z=e^{i\varphi}}\d\varphi
\]
making use of the above one-to-one cubic transformation of $h$.
\begin{proof}[Proof of Proposition~\ref{thm:main_Th_1}] We will rely on Lemma~\ref{lem:VDCP_lemma_delta_fixed}. Differentiating $\tilde{h}$ with respect to $\varphi$
\[
\frac{\partial^{2}\tilde{h}}{\partial\varphi^{2}}=\frac{\partial}{\partial\varphi}\left(\frac{\partial h}{\partial z}\frac{\partial z}{\partial\varphi}\right)=\frac{\partial^{2}h}{\partial z^{2}}\left(\frac{\partial z}{\partial\varphi}\right)^{2}+\frac{\partial h}{\partial z}\frac{\partial^{2}z}{(\partial\varphi)^{2}}
\]
it follows from \eqref{eq:scd_deriv_f_z_plus} and \eqref{eq:scd_derivat_f_z_minus}
that 
\[
\frac{\partial^{2}\tilde{h}}{\partial\varphi^{2}}(\varphi_{+},\alpha)=\frac{2\Im(z_{+})(1-\lambda^{2})\lambda}{\abs{1-\lambda z_{+}}^{4}}>0
\]
and 
\[
\frac{\partial^{2}\tilde{h}}{\partial\varphi^{2}}(\varphi_{-},\alpha)=-\frac{2\Im(z_{+})(1-\lambda^{2})\lambda}{\abs{1-\lambda z_{+}}^{4}}<0
\]
which shows that $J_{n}(k)$ is perfectly suited for applying the
approach in \cite[Section 2.3 p. 41]{VB} (with $x_{1}=\varphi_{-}$
and $x_{2}=\varphi_{+}$). Observe that 
\[
\abs{1-\lambda z_{\pm}}=\sqrt{\frac{1-\lambda^{2}}{\alpha}}
\]
and 
\begin{align*}
2\Im(z_{+}) & =2\sqrt{1-\left(\frac{\alpha(1+\lambda^{2})-(1-\lambda^{2})}{2\lambda\alpha}\right)^{2}}\\
 & =\frac{\sqrt{(1-\lambda^{2})(\alpha(1+\lambda)-(1-\lambda))((1+\lambda)-\alpha(1-\lambda))}}{\lambda\alpha}\\
 & =\frac{1-\lambda^{2}}{\lambda\alpha}\sqrt{\left(\alpha-\alpha_{0}\right)\left(\alpha_{0}^{-1}-\alpha\right)}.
\end{align*}
This gives 
\[
\frac{\partial^{2}\tilde{h}}{\partial\varphi^{2}}(\varphi_{\pm},\alpha)=\pm\alpha\sqrt{\left(\alpha-\alpha_{0}\right)\left(\alpha_{0}^{-1}-\alpha\right)}.
\]
A straightforward application of \cite[formula (2.36) p. 43]{VB}
gives 
\[
J_{n}(k)=\left(\frac{Ai(-n^{2/3}\gamma^{2})}{n^{1/3}}a_{0}+\frac{Ai'(-n^{2/3}\gamma^{2})}{n^{2/3}}a_{1}\right)\left(1+\cO(n^{-1/3})\right)
\]
where $a_{0}$ and $a_{1}$ are given (see \cite[formula (2.36c) p. 44]{VB})
by 
\begin{align*}
a_{0} & =\frac{\sqrt{\gamma}}{\sqrt{2}}\left(\frac{1}{\sqrt{\frac{\partial^{2}\tilde{h}}{\partial\varphi^{2}}(\varphi_{+},\alpha)}}+\frac{1}{\sqrt{\abs{\frac{\partial^{2}\tilde{h}}{\partial\varphi^{2}}(\varphi_{-},\alpha)}}}\right)\\
 & =\frac{\sqrt{2\gamma}}{\sqrt{\alpha}\left(\alpha-\frac{1}{\alpha_{0}^{-1}}\right)^{1/4}\left(\alpha_{0}^{-1}-\alpha\right)^{1/4}}\\
 & =\frac{(1-\lambda)^{1/4}}{\left(\lambda(1+\lambda)\right)^{1/12}}\frac{\sqrt{2}}{\sqrt{\alpha}\left(\alpha-\alpha_{0}\right)^{1/4}}
\end{align*}
and 
\[
a_{1}=\frac{\sqrt{\gamma}}{\sqrt{2}}\left(\frac{1}{\sqrt{\abs{\frac{\partial^{2}\tilde{h}}{\partial\varphi^{2}}(\varphi_{-},\alpha)}}}-\frac{1}{\sqrt{\frac{\partial^{2}\tilde{h}}{\partial\varphi^{2}}(\varphi_{+},\alpha)}}\right)=0.
\]
This yields

\[
J_{n}(k)=\frac{(1-\lambda)^{1/4}}{\left(\lambda(1+\lambda)\right)^{1/12}}\frac{\sqrt{2}}{\sqrt{\alpha}\left(\alpha-\alpha_{0}\right)^{1/4}}\frac{Ai(n^{2/3}\gamma^{2})}{n^{1/3}}\left(1+\cO(n^{-1/3})\right)
\]
and the result follows. \end{proof} 
We conclude our analysis with the discussion of the lower bound for $p=4$.
\begin{prop}
For $p=4$ we have 
\[
\Norm{B}{l_{4}^{A}}\asymp\left(\frac{\log n}{n}\right)^{1/4}.
\]
\end{prop}
\begin{proof}
The upper bound $\Norm{B}{l_{4}^{A}}\lesssim\left(\frac{\log n}{n}\right)^{1/4}$
is shown in Section 2. The proof of the lower bound will be concluded in
four steps. 

$\,$

\textbf{Step 1.} \textit{Application of }Theorem \ref{thm:main_Th_1}
\textit{and sommation over a suitable range of $k$. }

$\,$

Given $\lambda\in(0,1)$ we recall the notation we finally chose:
$\alpha_{0}=\frac{1-\lambda}{1+\lambda}$. We define $I_{n}$ to be
the following set of integers
\[
I_{n}=\mathbb{Z}\cap[\alpha_{0}^{-1}n-n^{3/4},\alpha_{0}^{-1}n-\sqrt{n}].
\]
A direct application of Theorem \ref{thm:main_Th_1} 
\[
\sum_{k\in I_{n}}\abs{\hat{B}(k)}^{4}\gtrsim\sum_{k\in I_{n}}\frac{(1-\lambda)}{\left(\lambda(1+\lambda)\right)^{1/3}}\frac{4}{\left(\frac{k}{n}\right)^{2}\left(\frac{k}{n}-\alpha_{0}\right)}\frac{(Ai(-n^{2/3}\gamma^{2}))^{4}}{n^{4/3}}
\]
where $\gamma^{2}=\frac{1-\lambda}{\left(\lambda(1+\lambda)\right)^{1/3}}\left(\alpha_{0}^{-1}-\frac{k}{n}\right).$
Plugging in the oscillatory behavior of $Ai$ 
\[
Ai(-x)=\frac{1}{x^{1/4}\sqrt{\pi}}\cos\left(\frac{2}{3}x^{3/2}-\frac{\pi}{4}\right)+\cO(x^{-7/4}),\qquad x\rightarrow+\infty
\]
we find
\begin{align*}
\sum_{k\in I_{n}}\abs{\hat{B}(k)}^{4} & \gtrsim\frac{1}{n^{4/3}}\sum_{k\in I_{n}}\frac{(1-\lambda)}{\left(\lambda(1+\lambda)\right)^{1/3}}\frac{4}{\left(\frac{k}{n}\right)^{2}\left(\frac{k}{n}-\alpha_{0}\right)}\frac{1}{n^{2/3}\gamma^{2}\pi^{2}}\cos^{4}\left(\frac{2}{3}n\gamma^{3/2}-\frac{\pi}{4}\right)\\
 & \gtrsim\frac{1}{n^{2}}\sum_{k\in I_{n}}\frac{1}{\gamma^{2}}\cos^{4}\left(\frac{2}{3}n\gamma^{3/2}-\frac{\pi}{4}\right).
\end{align*}
because $\lim_{n\rightarrow\infty}\frac{k}{n}=\alpha_{0}^{-1}$ and
$\lambda$ is fixed in (0,1). Plugging in the value of $\gamma$ this
yields
\begin{align*}
\sum_{k\in I_{n}}\abs{\hat{B}(k)}^{4} & \gtrsim\frac{1}{n^{2}}\sum_{k\in I_{n}}\frac{1}{\left(\alpha_{0}^{-1}-\frac{k}{n}\right)}\cos^{4}\left(\frac{2}{3}n\frac{(1-\lambda)^{3/2}}{\left(\lambda(1+\lambda)\right)^{1/2}}\left(\alpha_{0}^{-1}-\frac{k}{n}\right)^{3/2}-\frac{\pi}{4}\right)\\
 & =\frac{1}{n^{2}}\sum_{k\in I_{n}}\frac{1}{\left(\alpha_{0}^{-1}-\frac{k}{n}\right)}\cos^{4}\left(\pi s_{nk}-\frac{\pi}{4}\right)
\end{align*}
where $s_{nk}=\left\langle n\varphi(\frac{k}{n})\right\rangle $,
$\left\langle s\right\rangle $ denoting the fractional part of $s$
and $\varphi(t)=\frac{2}{3\pi}\frac{(1-\lambda)^{3/2}}{\left(\lambda(1+\lambda)\right)^{1/2}}\left(\alpha_{0}^{-1}-t\right)^{3/2}.$ 

$\,$

\textbf{Step 2.} \textit{Equidistribution of }$s_{nk}$. 

$\,$

Given $j\neq0$ in $\mathbb{Z}$ and $k$ in $I_{n}$ we consider
\[
A_{k}:=\sum_{l\in I_{n},\: l\leq k}\exp(2\pi ijs_{nk})=\sum_{l\in I_{n},\: l\leq k}\exp(2\pi ijn\varphi(\frac{k}{n})).
\]
To estimate the above sum we use one of Van der Corput's lemma \cite[Ch. 5, Lemma 4.6]{AZ}:
if $f''(x)\geq\mu>0$ or $f''(x)\leq-\mu<0$ on $[a,b]$ then
\[
\abs{\sum_{a<k\leq b}\exp(2\pi if(k))}\leq\left(\abs{f'(b)-f'(a)}+2\right)\left(\frac{4}{\sqrt{\mu}}+A\right)
\]
where $A$ is an absolute constant. We shall apply this lemma to
the case $f(x)=jn\varphi(\frac{x}{n})$, $b=\alpha_{0}^{-1}n-\sqrt{n}$
and $a=\alpha_{0}^{-1}n-n^{3/4}$. We have $f'(x)=j\varphi'(\frac{x}{n})$
and $f''(x)=\frac{j}{n}\varphi''(\frac{x}{n})$ where $\varphi'(t)=-\frac{1}{\pi}\frac{(1-\lambda)^{3/2}}{\left(\lambda(1+\lambda)\right)^{1/2}}\left(\alpha_{0}^{-1}-t\right)^{1/2}$
and $\varphi''(t)=\frac{1}{2\pi}\frac{(1-\lambda)^{3/2}}{\left(\lambda(1+\lambda)\right)^{1/2}}\left(\alpha_{0}^{-1}-t\right)^{-1/2}$.
In particular for $k\in I_{n}$ we have $\frac{1}{\sqrt{n}}\leq\alpha_{0}^{-1}-\frac{k}{n}\leq\frac{1}{n^{1/4}}$
so that $f'(a)=o(1)$, $f'(b)=o(1)$ as $n$ tends to $\infty$ and
\[
f''(x)\geq\frac{j}{2\pi}n^{-7/8}\frac{(1-\lambda)^{3/2}}{\left(\lambda(1+\lambda)\right)^{1/2}},\qquad x\in[a,b],\: j\geq1,
\]
\[
f''(x)\leq-\frac{\abs{j}}{2\pi}n^{-7/8}\frac{(1-\lambda)^{3/2}}{\left(\lambda(1+\lambda)\right)^{1/2}},\qquad x\in[a,b],\: j\leq-1.
\]
For $n$ large enough we obtain that for any $k\in I_{n}$ 
\[
\abs{A_{k}}\lesssim n^{7/16}.
\]

$\,$

\textbf{Step 3.} \textit{Approximation by trigonometric polynomials
and Abel's transformation}.

$\,$

We reproduce and adapt the proof of \cite[Lemma 3]{DG} 
We define $g(x)=\cos^{4}\left(\pi\left(x-\frac{1}{4}\right)\right)$.
Our aim is to prove that there exists a limit: 
\[
\lim_{n\rightarrow\infty}\frac{1}{n\log n}\sum_{k\in I_{n}}\frac{g(s_{nk})}{\left(\alpha_{0}^{-1}-\frac{k}{n}\right)}=\int_{0}^{1}g(x){\rm d}x.
\]
 We observe that $g$ is continuous and 1-periodic on the real line.
In particular according to Fejér's Theorem, for any $\epsilon>0$
there is a trigonometric polynomial 
\[
p(x)=\sum_{\abs{j}\leq N}c_{j}\exp(2\pi ijx)
\]
such that for $x\in[0,1]$ $-\epsilon\leq g(x)-p(x)\leq\epsilon$. 

We put $M=\sum_{k\in I_{n}}\frac{1}{\left(\alpha_{0}^{-1}-\frac{k}{n}\right)}$.
Using the monotonicity and positivity of $t\mapsto\frac{1}{\alpha_{0}^{-1}-t}$
we compare $M$ with an integral to obtain $M\sim_{n\rightarrow\infty}n\int_{\alpha_{0}^{-1}-n^{-1/4}}^{\alpha_{0}^{-1}-n^{-1/2}}\frac{{\rm d}t}{\left(\alpha_{0}^{-1}-t\right)}$.
After a change of variable we find $M\sim_{n\rightarrow\infty}n\int_{n^{-1/2}}^{n^{-1/4}}\frac{{\rm d}u}{u}$
that is to say
\[
M\sim_{n\rightarrow\infty}\frac{n\log n}{4}.
\]
We now write
\[
M^{-1}\sum_{k\in I_{n}}\frac{p(s_{nk})}{\left(\alpha_{0}^{-1}-\frac{k}{n}\right)}=\sum_{\abs{j}\leq N}c_{j}M^{-1}\sum_{k\in I_{n}}\frac{1}{\left(\alpha_{0}^{-1}-\frac{k}{n}\right)}\exp(2\pi ijs_{nk})
\]
and by applying Abel's summation formula we get for $j\neq0$
\begin{align*}
\sum_{k\in I_{n}}\frac{1}{\left(\alpha_{0}^{-1}-\frac{k}{n}\right)}\exp(2\pi ijs_{nk}) & =\sum_{k\in I_{n}}\frac{1}{\left(\alpha_{0}^{-1}-\frac{k}{n}\right)}(A_{k}-A_{k-1})\\
 & =\left[\frac{A_{[\alpha_{0}^{-1}n-\sqrt{n}]}}{\left(\alpha_{0}^{-1}-\frac{[\alpha_{0}^{-1}n-\sqrt{n}]+1}{n}\right)}-\frac{A_{[\alpha_{0}^{-1}n-n^{3/4}]}}{\left(\alpha_{0}^{-1}-\frac{[\alpha_{0}^{-1}n-n^{3/4}]+1}{n}\right)}\right]\\
 & -\sum_{k\in I_{n}}A_{k}\left(\frac{1}{\left(\alpha_{0}^{-1}-\frac{k+1}{n}\right)}-\frac{1}{\left(\alpha_{0}^{-1}-\frac{k}{n}\right)}\right).
\end{align*}
This yields 
\begin{align*}
M^{-1}\sum_{k\in I_{n}}\frac{1}{\left(\alpha_{0}^{-1}-\frac{k}{n}\right)}\exp(2\pi ijs_{nk}) & =\cO\left(\frac{\sqrt{n}n^{7/16}}{n\log(n)}\right)+\cO\left(\frac{n^{1/4}n^{7/16}}{n\log(n)}\right)\\
 & +\cO\left(\frac{n^{7/16}}{n\log(n)}\right)\sum_{k\in I_{n}}\left(\frac{1}{\left(\alpha_{0}^{-1}-\frac{k+1}{n}\right)}-\frac{1}{\left(\alpha_{0}^{-1}-\frac{k}{n}\right)}\right)\\
 & =\cO\left(\frac{1}{n^{1/16}\log(n)}\right)\\
 & +\cO\left(\frac{n^{7/16}}{n\log(n)}\right)\left[\frac{1}{\left(\alpha_{0}^{-1}-\frac{[\alpha_{0}^{-1}n-\sqrt{n}]+1}{n}\right)}-\frac{1}{\left(\alpha_{0}^{-1}-\frac{[\alpha_{0}^{-1}n-n^{3/4}]+1}{n}\right)}\right]\\
 & =\cO\left(\frac{1}{n^{1/16}\log(n)}\right).
\end{align*}
In particular
\[
\lim_{n\rightarrow\infty}M^{-1}\sum_{k\in I_{n}}\frac{p(s_{nk})}{\left(\alpha_{0}^{-1}-\frac{k}{n}\right)}=c_{0}=\int_{0}^{1}p(x){\rm d}x
\]
but by construction on the interval $[0,1]$ $p-\epsilon\leq g\leq\epsilon+p$
and therefore
\[
M^{-1}\sum_{k\in I_{n}}\frac{p(s_{nk})}{\left(\alpha_{0}^{-1}-\frac{k}{n}\right)}-\epsilon\leq M^{-1}\sum_{k\in I_{n}}\frac{g(s_{nk})}{\left(\alpha_{0}^{-1}-\frac{k}{n}\right)}\leq M^{-1}\sum_{k\in I_{n}}\frac{p(s_{nk})}{\left(\alpha_{0}^{-1}-\frac{k}{n}\right)}+\epsilon.
\]
Passing after to the limit as $n$ tends to $\infty$ we get
\[
\int_{0}^{1}p(x){\rm d}x-\epsilon\leq\limsup_{n\rightarrow\infty}M^{-1}\sum_{k\in I_{n}}\frac{g(s_{nk})}{\left(\alpha_{0}^{-1}-\frac{k}{n}\right)}\leq\int_{0}^{1}p(x){\rm d}x+\epsilon.
\]
Substracting $\int_{0}^{1}g(x){\rm d}x$ this yields 
\[
\limsup_{n\rightarrow\infty}M^{-1}\sum_{k\in I_{n}}\frac{g(s_{nk})}{\left(\alpha_{0}^{-1}-\frac{k}{n}\right)}-\int_{0}^{1}g(x){\rm d}x\leq\int_{0}^{1}\left(p(x)-g(x)\right){\rm d}x+\epsilon\leq2\epsilon
\]
on one hand and 
\[
-2\epsilon\leq\int_{0}^{1}\left(p(x)-g(x)\right){\rm d}x-\epsilon\leq\limsup_{n\rightarrow\infty}M^{-1}\sum_{k\in I_{n}}\frac{g(s_{nk})}{\left(\alpha_{0}^{-1}-\frac{k}{n}\right)}-\int_{0}^{1}g(x){\rm d}x
\]
on the other hand. Since $\epsilon$ is arbitrarily small we obtain
$\limsup_{n\rightarrow\infty}M^{-1}\sum_{k\in I_{n}}\frac{g(s_{nk})}{\left(\alpha_{0}^{-1}-\frac{k}{n}\right)}=\int_{0}^{1}g(x){\rm d}x$
and in the same way $\liminf_{n\rightarrow\infty}M^{-1}\sum_{k\in I_{n}}\frac{g(s_{nk})}{\left(\alpha_{0}^{-1}-\frac{k}{n}\right)}=\int_{0}^{1}g(x){\rm d}x$.
The result follows.

$\,$

\textbf{Step 4.} \textit{Conclusion}

$\,$

With the result from Step 3 and going back to Step 1

\begin{align*}
\sum_{k\in I_{n}}\abs{\hat{B}(k)}^{4} & \gtrsim\frac{1}{n^{2}}\sum_{k\in I_{n}}\frac{1}{\left(\alpha_{0}^{-1}-\frac{k}{n}\right)}\cos^{4}\left(\pi s_{nk}-\frac{\pi}{4}\right)\\
 & =\frac{1}{n^{2}}\sum_{k\in I_{n}}\frac{g(s_{nk})}{\left(\alpha_{0}^{-1}-\frac{k}{n}\right)}\gtrsim\frac{\log n}{n}
\end{align*}
which completes the proof. \end{proof}

\end{document}